\documentclass[12pt]{article}   	
\usepackage{geometry}                		
\usepackage{graphicx}				
\usepackage{caption}
\usepackage{subcaption}
\usepackage{amssymb}
\usepackage{amsthm}
\usepackage{csquotes}
\usepackage{amsmath}
\usepackage{mathtools}
\usepackage{float}
\usepackage{physics}
\MakeOuterQuote{"}
\newtheorem*{theorem*}{Theorem}
\newtheorem{definition}{Definition}
\newtheorem{theorem}{Theorem}[section]
\newtheorem{proposition}[theorem]{Proposition}
\newtheorem{lemma}[theorem]{Lemma}

\newtheorem*{conjecture*}{Conjecture}
\title{Stationary Solutions of the Curvature Preserving Flow on Space Curves}
\author{Matei P. Coiculescu}
\begin{document}
\maketitle
\begin{abstract}
We study a geometric flow on curves, immersed in $\mathbb{R}^3$, that have strictly positive torsion. The evolution equation is given by
$$X_{t}=\frac{1}{\sqrt{\tau}} \textbf{B}$$
where $\tau$ is the torsion and $\textbf{B}$ is the unit binormal vector. In the case of constant curvature, we find all of the stationary solutions and linearize the PDE for torsion around stationary solutions admitting an explicit formula. Afterwards, we prove the $L^2(\mathbb{R})$ linear stability of the stationary solutions corresponding to helices with constant curvature and constant torsion.
\end{abstract}
\section{Introduction}
Substantial work has been done towards understanding geometric flows on curves immersed in Riemannian Manifolds. For example, the author of \cite{G} proves that the mean-curvature flow shrinks embedded curves in the plane to a point in finite time, becoming round in the limit. Also, it is found in \cite{H} that the vortex filament flow is equivalent to the non-linear Schr\"odinger equation, which enables the discovery of explicit soliton solutions. Recently, geometric evolutions that are integrable, in the sense of admitting a Hamiltonian structure, have also been of interest. The authors of \cite{I} and \cite{BSW} analyze the integrability of flows in Euclidean space and Riemannian Manifolds, respectively. In this note, we study the following geometric flow for curves in $\mathbb{R}^3$ with strictly positive torsion that preserves arc-length and curvature:
$$X_t=\frac{1}{\sqrt{\tau}}\textbf{B}.$$ 
Hydrodynamic and magnetodynamic motions related to this geometric evolution equation have been considered, as mentioned in \cite{SR}. The case when curvature is constant demonstrates a great deal of structure, as we shall soon see. After rescaling so that the curvature is identically $1$, the evolution equation for torsion is given by:
$$\tau_t = D_s\big(\tau^{-1/2}-\tau^{3/2}+D_s^2(\tau^{-1/2})\big).$$
This flow has been studied since at least the publication of \cite{SR}. The authors of \cite{SR} show that the above evolution equation, which they term the \textit{extended Dym equation}, is equivalent to the m$^2$KDV equation. In addition, they present auto-B\"acklund transformations and compute explicit soliton solutions. We hope to continue the investigation of the curvature-preserving geometric flow:
\begin{itemize}
\item We characterize the flow $X_t=\frac{1}{\sqrt{\tau}}\textbf{B}$ as the unique flow on space curves that is both curvature and arc-length preserving.
\item We provide another proof that this flow is equivalent to the m$^2$KDV equation, and by doing so, we find the first two conserved densities of the flow
$$\int \sqrt{\tau} ds \textrm{ and } \int \tau ds$$
which are the same as for the KDV equation. 
\item We find all stationary solutions to the geometric flow in the case of constant curvature, including a two-parameter family of explicit solutions.
\item We derive the linearization of the evolution equation for torsion around explicit stationary solutions, and prove $L^2(\mathbb{R})$ stability of the linearization in the case of constant torsion (or for helices).
\end{itemize}
A tedious calculation yields the fact that the evolution in the case of non-constant curvature is not integrable, even in the formal sense of \cite{MSS}. Although we do not pursue this here, it might be interesting to study the case of non-constant curvature, especially for curves with almost constant curvature.

We would like to thank Richard Schwartz and Benoit Pausader for helpful discussions about this topic. We would also like to thank Wolfgang Schief for pointing out his joint paper \cite{SR} with C. Roger.

\section{Preliminaries}
We recall the following standard computation:
\begin{lemma}[cf. \cite{BSW}, \cite{I}]
Let $\gamma(t)$ be a family of smooth curves immersed in $\mathbb{R}^3$ and let $X(s,t)$ be a parametrization of $\gamma(t)$ by arc-length. Consider the following geometric evolution equation:
\begin{equation}X_t = h_1 \textbf{T} +h_2 \textbf{N}+h_3 \textbf{B}\end{equation}
where $\{\textbf{T},\textbf{N},\textbf{B}\}$ is the Frenet-Serret Frame and where we denote the curvature and torsion by $\kappa$ and $\tau$, respectively. Let $h_1, h_2, h_3$ be arbitrary smooth functions of $\kappa$ and $\tau$ on $\gamma(t)$. If the evolution is also arc-length preserving, then the evolution equations of $\kappa$ and $\tau$ are
$$\begin{pmatrix} \kappa_t \\ \tau_t \end{pmatrix}=P\begin{pmatrix} h_3 \\ h_1 \end{pmatrix}$$
where $P$ is
$$\begin{pmatrix} -\tau D_s - D_s\tau & D_s^2\frac{1}{\kappa}D_s-\frac{\tau^2}{\kappa}D_s+D_s\kappa\\
D_s\frac{1}{\kappa}D_s^2-D_s\frac{\tau^2}{\kappa}+\kappa D_s& D_s(\frac{\tau}{\kappa^2}D_s+D_s\frac{\tau}{\kappa^2})D_s+\tau D_s+D_s\tau\end{pmatrix}.$$
\end{lemma}
The next theorem follows without difficulty.
\begin{theorem}
Up to a rescaling, a geometric evolution of curves immersed in $\mathbb{R}^3$, as in equation $(1)$, is both curvature and arc-length preserving if and only if its evolution evolution is equivalent to 
\begin{equation}
X_t = \frac{1}{\sqrt{\tau}}\textbf{B}.
\end{equation}
\end{theorem}
\begin{proof}
Let $X_t$ be a curvature and arc-length preserving geometric flow. The tangential component of $X_t$ in equation $(1)$ provides no interesting geometric information; it amounts to a re-parametrization of the curve. Thus, we may assume that $h_1=0$, so, since $X_t$ is arc-length preserving, $h_2=0$ as well. Lemma 2.1 gives us that 
$$\kappa_t = -\tau D_s (h_3)-D_s(\tau h_3)=-2\tau D_s (h_3)-h_3D_s(\tau).$$
Since $X_t$ is curvature preserving, we must have $\kappa_t=0$, or 
$$D_s(\log{h_3})=D_s(\log{\tau^{-1/2}}).$$
Integrating, we see that 
$$h_3=\frac{c}{\sqrt{\tau}}$$
where $c$ is a constant. Therefore, up to a rescaling, the evolution $X_t$ must be precisely as in $(2)$.
\end{proof}
Unfortunately, the evolution in $(2)$ only makes sense if $\gamma(t)$ has strictly positive torsion (or strictly negative torsion, with the flow $X_t=\frac{1}{\sqrt{-\tau}}\textbf{B}$). This motivates the following definition.
\begin{definition}
We call a smooth curve immersed in $\mathbb{R}^3$ a positive curve and only if it has strictly positive torsion.
\end{definition}
Fortunately, there are many interesting positive curves. For example, some knots admit a parametrization with constant curvature and strictly positive torsion, and there exist closed curves with constant (positive) torsion. Henceforth, we will only consider positive curves. 

The partial differential equation governing the evolution of $\tau$ follows below:
\begin{lemma}
For the geometric flow given in equation $(2)$, the evolution equations of curvature and torsion are $\kappa_t =0$ and
\begin{equation}
\tau_t = \kappa D_s(\tau^{-1/2})+D_s\bigg(\frac{D_s^2(\tau^{-1/2})-\tau^{3/2}}{\kappa}\bigg)
\end{equation}
\end{lemma}
The equation for torsion is reminiscent of the Rosenau-Hyman family of equations, which are studied, inter alia, in \cite{HEER} and \cite{LW}. The authors of \cite{SR} call the evolution of the torsion, in the case of constant curvature, the \textit{extended Dym equation} for its relationship with the Dym equation (a rescaling and limiting process converts the evolution for torsion to the Dym equation). 
As discussed in \cite{BSW}, the condition for the flow $X_t$ from Lemma 2.1 to be the gradient of a functional is for the Frechet derivative of $(h_3,h_1)$ to be self-adjoint. In general, this does not occur for the flow under our consideration in equation $(2)$, so it cannot be integrable in the sense of admitting a Hamiltonian structure (Indeed, when $\kappa$ is not constant, it is not even formally integrable in the sense of \cite{MSS}). Nevertheless, the case of constant curvature exhibits a great deal of structure, which makes the study of the evolution equation in $(3)$ worthwhile.
\section{Constant Curvature}
When $\kappa$ is constant, we may rescale the curves $\gamma(t)$ so that $\kappa \equiv 1$. In this way, the evolution of torsion becomes
\begin{equation}
\tau_t = D_s\big(\tau^{-1/2}-\tau^{3/2}+D_s^2(\tau^{-1/2})\big)
\end{equation}
\subsection{Equivalence with the m$^2$KDV Equation}
We recall the notion of "equivalence" of two partial differential equations from \cite{CFA}:
\begin{definition}
Two partial differential equations are equivalent if one can be obtained from the other by a transformation involving the dependent variables or the introduction of a potential variable.
\end{definition}
The authors in \cite{CFA} discuss a general method of transforming quasilinear partial differential equation, such as the evolution of $\tau$ in $(4)$, to semi-linear equations. By applying their algorithm, we obtain another proof of the following theorem, first demonstrated in \cite{SR}.
\begin{theorem}[\cite{SR}]
The evolution equation for torsion, in the case of constant curvature, which is given by
$$\tau_t = D_s\big(\tau^{-1/2}-\tau^{3/2}+D_s^2(\tau^{-1/2})\big),$$
is equivalent to the m$^2$KDV equation. Thus, it is a completely integrable evolution equation.
\end{theorem}
\begin{proof}
First, let $\tau=v^2$, so that $(4)$ becomes
$$
2vv_t=D_s\bigg(\frac{1}{v}-v^3+D_s^2\big(\frac{1}{v}\big)\bigg)=-\frac{v_s}{v^2}-3v^2v_s-\frac{v_{sss}}{v^2}-\frac{6v_s^3}{v^4}+\frac{6v_sv_{ss}}{v^3}$$
or, in a simpler form:
\begin{equation}
v_t=D_s\bigg(\frac{1}{4v^2}-\frac{3v^2}{4}+\frac{3v_s^2}{4v^4}-\frac{v_{ss}}{2v^3}\bigg)
\end{equation}
A potentiation, $v=w_s$ followed by a simple change of variables $(t\rightarrow -t/2)$ yields
\begin{equation}
w_t=-\frac{1}{2w_s^2}+\frac{3w_s^2}{2}-\frac{3w_{ss}^2}{2w_s^4}+\frac{w_{sss}}{w_s^3}.
\end{equation}
Equation $(6)$ is fecund territory for a pure hodograph transformation, as used, for example, in \cite{CFA}. Let $\tilde{t}=t$, $\xi=w(s,t)$, and $s=\eta(\xi,\tilde{t})$. The resulting equation, after a simple computation, is 
\begin{equation}
\eta_{\tilde{t}}=\eta_{\xi\xi\xi}-\frac{3\eta_{\xi\xi}^2}{2\eta_\xi}+\frac{\eta_\xi^3}{2}-\frac{3}{2\eta_\xi}.
\end{equation}
We rename the variables to the usual variables of space and time: $s$ and $t$; in addition, we anti-potentiate the equation by letting $\eta_s=z$. This makes equation $(7)$ equivalent to
\begin{equation}
z_t = z_{sss}-\frac{3}{2}\bigg(\frac{z_s^2}{z}\bigg)_s+\frac{3z^2z_s}{2}+\frac{3z_s}{2z^2}.
\end{equation}
Lastly, if we let $z=e^u$ and simplify, equation $(8)$ becomes
\begin{equation}
u_t=u_{sss}-\frac{u_s^3}{2}+\frac{3u_s}{2}\big(e^{2u}+e^{-2u}\big)
\end{equation}
which is nothing but the Calogero-Degasperis-Fokas equation, for a particular choice of parameters. The CDF equation is linearizable and has been studied since at least the publication of \cite{CD}. There is a Miura-type transformation between the CDF equation and the m$^2$KDV equation (see \cite{CFA}, \cite{HEER}), so the evolution equation for torsion, in the case of constant curvature, is equivalent to the m$^2$KDV equation, as desired.
\end{proof}
By Theorem 3.1, we are guaranteed long term existence for our geometric flow. In other words, if we begin with a sufficiently nice positive curve with constant curvature and begin evolving it according to $(4)$, then the torsion and all of its derivatives remain bounded and the torsion remains strictly positive for all time. 

Equations $(4)$ and $(5)$ give us the first two integrals of motion of this flow:
$$\int \sqrt{\tau} ds \textrm{ and } \int \tau ds.$$
The rest can be found by pulling back the m$^2$KDV invariants. These invariants were obtained in a different way by the authors of \cite{SR}.

\subsection{Stationary Solutions}
Helices, with constant curvature and constant torsion, are the obvious stationary solutions. In what follows, we find the rest. 
\begin{theorem}
The stationary solutions of 
$$\tau_t = D_s\big(\tau^{-1/2}-\tau^{3/2}+D_s^2(\tau^{-1/2})\big)$$
are given by the following integral formula
$$\int\frac{du}{\sqrt{C+2Au-u^2-u^{-2}}}=s+k$$
where $\tau(s)=u(s)^{-2}$ and where $A,k$ and $C$ are appropriate real constants. When $A=0$ we get an explicit formula for the solutions:
$$\tau(s)=\frac{2}{C\pm\sqrt{(-4+C^2)}\cdot\sin(2(s+k))},$$
with $k$ and  $C\geq 2$ real constants.
\end{theorem}
\begin{proof}
Let $\tau=u^{-2}$, then, after integrating once, we must examine the following ordinary differential equation (where $A$ is a constant):
\begin{equation}
A=u-\frac{1}{u^3}+D_s^2(u)
\end{equation}
Since equation $(10)$ is autonomous, we may proceed with a reduction of order argument. Let $w(u)=D_s(u)$ so that $D_s^2(u)=wD_u(w)$ by the chain rule. This substitution gives us the first order equation:
\begin{equation}
wD_u(w)=A-u+\frac{1}{u^3}
\end{equation}
or
$$D_u(w^2)=2A-2u+\frac{2}{u^3}.$$
Integrating, we get
$$D_s(u)=w(u)=\sqrt{C+2Au-u^2-u^{-2}}$$
which is a separable differential equation. So, the stationary solutions of equation $(4)$ are given by the following integral formula:
\begin{equation}
\int\frac{du}{\sqrt{C+2Au-u^2-u^{-2}}}=s+k
\end{equation}
for appropriate constants $C, A, \textrm{ and } k$.
It would be pleasant to have explicit solutions, and this occurs in the case when $A=0$, which is more easily handled. Equation $(11)$ above becomes
$$D_s(u)=w(u)=\sqrt{C-u^2-u^{-2}}$$
which is a differential equation that can be solved with the aid of Mathematica or another computer algebra system. The result is
\begin{equation}
u(s)=\sqrt{\frac{C\pm\sqrt{(-4+C^2)}\cdot\sin(2(s+k))}{2}}
\end{equation}
Where $C,k$ are real constants and $C\geq 2$. The corresponding torsion is:
$$\tau(s)=\frac{2}{C\pm\sqrt{(-4+C^2)}\cdot\sin(2(s+k))}$$
\end{proof}
Integrating $\tau$ and $\kappa$ as above using the Frenet-Serret equations will yield the corresponding stationary curves, up to a choice of the initial Frenet-Serret frame and isometries of $\mathbb{R}^3$.  
\subsection{$L^2(\mathbb{R})$ Linear Stability of Helices}

First, we derive the linearization of the evolution for torsion around the stationary solutions corresponding to helices. The linearization of equation $(4)$ at any stationary solution $\tau_0$ is obtained by letting $\tau(s,t) = \tau_0(s,t)+\epsilon w(s,t)$, substituting into equation $(4)$, dividing by $\epsilon$ and then taking the limit as $\epsilon\rightarrow 0$. This is nothing more than the Gateaux derivative of our differential operator at $\tau_0$ in the direction of $w$. Alternatively, one may think of $w$ as the first-order approximation for solutions of $(4)$ near the stationary solution. We can perform this operation when $\tau_0$ is given by an explicit formula, but for brevity's sake, we only mention here the linearization around helices when $\tau_0$ is constant. A short calculation yields
\begin{proposition}
The linearization of the evolution equation $(4)$ around the stationary solutions of constant torsion is
\begin{equation}w_t+2w_s+\frac{1}{2}w_{sss}=0.\end{equation}
\end{proposition}
In what follows, we show the $L^2(\mathbb{R})$ linear stability of the the constant torsion stationary solution. First we recall the definition of linear stability:
\begin{definition}
A stationary solution $\phi$ of a nonlinear PDE is called $L^2(\mathbb{R})$ linearly stable when $v = 0$ is a stable solution of the corresponding linearized PDE with respect to the $L^2(\mathbb{R})$ norm and whenever $v_{t=0}$ is in $L^2(\mathbb{R})$.
\end{definition}
To work towards this, we need to use test functions from the Schwartz Space $\mathcal{S}(\mathbb{R})$, so we first recall that 
$$\mathcal{S}(\mathbb{R}):=\{ f\in C^{\infty}(\mathbb{R}) \textrm{ s.t. } \sup_{x\in\mathbb{R}} (1+\abs{x})^N\abs{\partial^{\alpha}f(x)}<\infty, \textrm{ for all } N, \alpha \}.$$
Intuitively, functions in $\mathcal{S}(\mathbb{R})$ are smooth and rapidly decreasing. For more on distributions and the Schwartz Space, review \cite{F}.
The rest of this section is devoted to proving:
\begin{theorem}
Helices correspond to $L^2(\mathbb{R})$ linearly stable stationary solutions of $$\tau_t = D_s\big(\tau^{-1/2}-\tau^{3/2}+D_s^2(\tau^{-1/2})\big).$$
\end{theorem}
\begin{proof}
Helices correspond to the constant torsion stationary solutions, so we analyze the linearized PDE in (14):
$$w_t+2w_s+\frac{1}{2}w_{sss}=0.$$
Henceforth, $w_0(s)$ denotes the initial data and $w(s,t)$ denotes the respective solution to the above, linear PDE. Thus, to prove the theorem, it suffices to show: for every $\epsilon$, there exists a $\delta$ such that if $w_0\in L^2(\mathbb{R})$ and $\|w_0\|_2<\delta$, then $\|w(t)\|_2<\epsilon$ for all $t\geq0$. 

We follow the standard process of finding weak solutions to linear PDEs via the Fourier transform. Moreover, we know by the Plancherel Theorem that we can extend the Fourier transform by density and continuity from $\mathcal{S}(\mathbb{R})$ to an isomorphism on $L^2(\mathbb{R})$ with the same properties. Hence, it suffices to prove the desired stability result for initial data in $\mathcal{S}(\mathbb{R})$. 

Let $F_t(\xi)=e^{4 i \pi ^3 \xi^3 t-4 i \pi  \xi t}$. We notice that since $F_t$ is a bounded continuous function for all $t\geq 0$, it can be considered a tempered distribution (or a member of $\mathcal{S}'(\mathbb{R})$, the continuous linear functionals on $\mathcal{S}(\mathbb{R})$), so its inverse Fourier transform makes sense. 

Indeed, we can let $B_t(s)=\mathcal{F}^{-1}(F_t(\xi))\in \mathcal{S}'(\mathbb{R})$ and, again, we denote $w_0\in \mathcal{S}(\mathbb{R})$ to be our initial data. Let
$$w(t)=B_t \ast w_0$$
so that $w(t)$ is a $C^{\infty}$ function with at most polynomial growth for all of its derivatives (see \cite{F}). Moreover, $w(t)$ satisfies equation $(14)$ in the distributional sense, as can be checked by taking the Fourier Transform. Lastly, since the Fourier Transform is a unitary isomorphism, it follows that
$$\lim_{t\rightarrow 0} w(t) = w_0$$
in the distribution topology of $\mathcal{S}'(\mathbb{R})$. Hence, $w(t)$ is the weak solution to equation $(14)$ with initial data $w_0\in\mathcal{S}(\mathbb{R})$. In our final step, we use the Plancherel Theorem and the fact that $\mathcal{F}(B_t)=F_t$ is a continuous function with $\|F_t\|_\infty =1$ for all $t\geq0$ to get:
$$\|w(t)\|_2=\|B_t \ast w_0 \|_2 =\|F_t \cdot \mathcal{F}(w_0)\|_2\leq \|F_t\|_\infty \|\mathcal{F}(w_0)\|_2= \|\mathcal{F}(w_0)\|_2=\|w_0\|_2.$$
From the inequality above, the desired $L^2(\mathbb{R})$ stability for initial data in $\mathcal{S}(\mathbb{R})$ follows forthrightly.
\end{proof}
\subsection{Numerical Rendering of a Stationary Curve}
We provide here the figures obtained from a numerical integration of the Frenet-Serret equations on Mathematica for the following choice of torsion:
$$\tau_1(s)=\frac{2}{3+\sqrt{5}\sin(2s)}$$ 
\begin{figure}[H]
\centering
\includegraphics[width=0.5\textwidth]{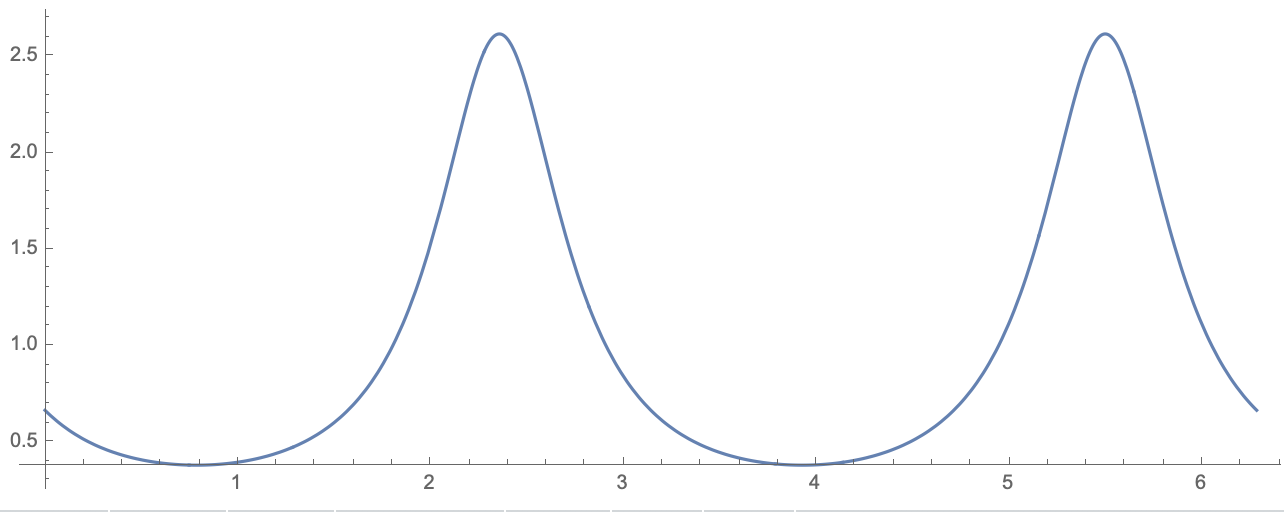}
\caption{This is the graph of the torsion $\tau_1$ over two periods.}
\end{figure}
\begin{figure}[H]
\centering
\includegraphics[width=0.5\textwidth]{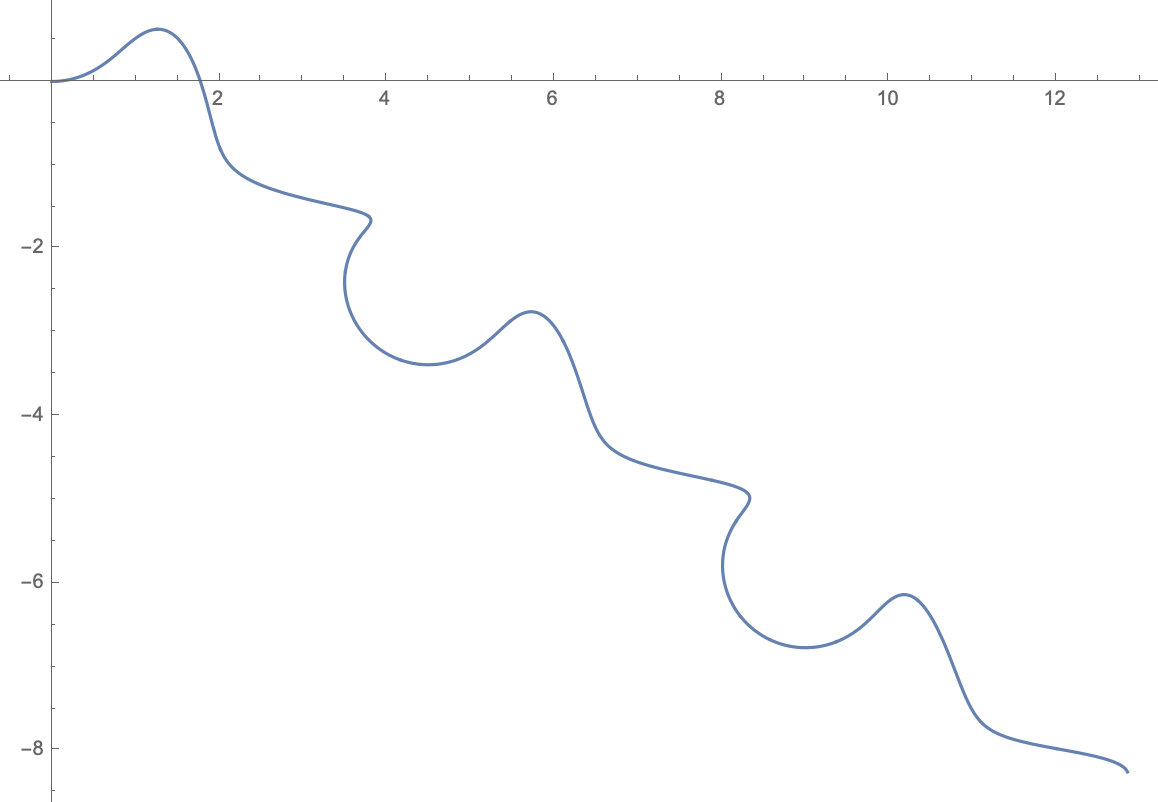}
\caption{This is the projection of the curve corresponding to $\tau_1$ into the $xy$ plane. The projections into the other planes look very similar.}
\end{figure}
\begin{figure}[H]
\centering
\includegraphics[width=0.5\textwidth]{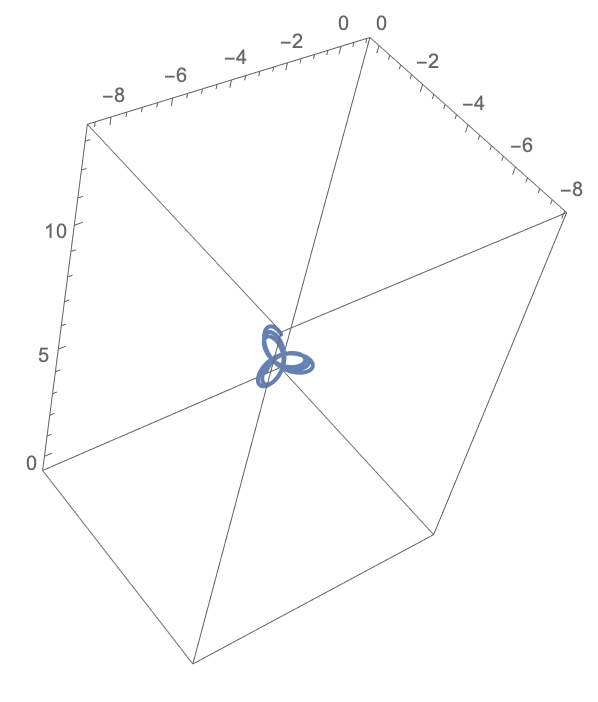}
\caption{This is a top-down view of the same curve, now exhibiting an almost trefoil shape.}
\end{figure}

\end{document}